\documentclass[12pt]{article}

\usepackage{mathtools}
\usepackage{tikz}
\usepackage{amsmath,amsfonts,amssymb,amsthm,bbm} 
\usepackage{graphicx} 
\usepackage{geometry} 
\usepackage{titlesec} 
\usepackage{enumitem} 
\usepackage{hyperref} 
\usepackage{lipsum} 
\newcommand{\prob}{\mathbb{P}}
\newcommand{\E}{\mathbb{E}}

\newcommand{\ind}{\mathbbm{1}}
\newtheorem{theorem}{Theorem}[section]

\newtheorem{example}[theorem]{Example}
\newtheorem{remark}[theorem]{Remark}

\titleformat{\section}{\normalfont\large\bfseries}{\thesection}{1em}{\MakeUppercase}
\titleformat{\subsection}{\normalfont\normalsize\bfseries}{\thesubsection}{1em}{}
\titleformat{\subsubsection}[runin]{\normalfont\normalsize\bfseries}{\thesubsubsection}{1em}{}[:]

\setlist[enumerate,1]{label=\arabic*.}
\setlist[enumerate,2]{label=(\alph*)}
\setlist[itemize,1]{label=$\bullet$}

\title{\textbf{\Large Finite Time Explosion of Stochastic Differential Equations: A survey into Khasminskii's Lyapunov Method and its Consistency with the Osgood Criterion}}
\author{\textsc{Seungsoo Lee}}
\date{\today}

\begin{document}

\maketitle

\begin{abstract}
\noindent Solutions of Stochastic Differential Equations can have three types of explosive behaviors: almost-sure non-explosive, explosion with positive probability, and almost sure explosion. In this paper, we will provide a survey of Khasminskii's Lyapunov method for classifying explosive behaviors of solutions of stochastic differential equations. We will embark our expedition by examining the renowned Feller's test for explosion and observing its shortfalls. Afterwards, we will present Khasminskii's Lyapunov method for almost-sure non-explosion, explosion with positive probability, and almost-sure explosion. Ample examples will be provided to illuminate the power of Khasminskii's Lyapunov methods. Furthermore, quick layovers will be made to extend Khasminskii's Lyapunov method for almost-sure non-explosion and explosion with positive probability for jump processes with constant Poisson intensities. 
    
\end{abstract}

\section{Background}

In this paper, we will go over the established methods to analyze explosive behaviors for stochastic differential equations and their advantages and disadvantages. Namely, we will look at the renowned Feller's test for explosion, and the lesser known Lyapunov function method. First show the explosive behavior for the deterministic ordinary differential equations (ODE) to provide a motivation for studying the explosive behavior of stochastic differential equations (SDE). 

It has already been shown in \cite{osgood1898beweis}
that finite time explosion in ODE has already been fully characterized by the Osgood condition which states that for an ODE of the form $\frac{dx}{dt}=b(x(t))>0$ with $x(0)=\xi>0$ and $t>0$ explodes if and only if 
$$\int_\xi^\infty \frac{1}{b(s)}ds<\infty$$
However, when our differential equation is perturbed with noise, our characterization becomes more complicated as the solution path can explode with positive probability, probability 1, and not explode with probability 1. It has already been shown in \cite{leon2011osgood} that when we have noise with constant intensity, that the Osgood condition of the drift parameter dictates the almost sure explosion and almost sure non-explosion of solution paths. Yet, we can see that the Osgood condition is no longer applicable when we have noise with high powers as shown in Example 1.2, where the drift indicates explosion, but the addition of large noise prevents explosion. Thus, further analysis is required to characterize the explosive behavior of SDEs. 

Throughout the paper, we shall define a solution of our SDE to be explosive when the hitting time for our solution, $X_t$, is finite, $\tau_\infty<\infty$, and non-explosive if $\tau_\infty=\infty$. Thus, there are three possible cases involving explosions: positive probability of explosion $\prob(\tau_\infty<\infty)>0$, almost sure non-explosion $\prob(\tau_\infty=\infty)=1$, and almost sure explosion $\prob(\tau_\infty<\infty)=1$. Of the three methods for testing explosion, only the Lyapunov function method will fully explore all three types of explosions.  

One prominent test for finite time explosion is Feller's test for explosion as shown by Karatzas and Shreve \cite{karatzas2014brownian}
\begin{theorem} Let our SDE be of the form
$$dX_t=b(X_t)dt+\sigma(X_t)dW_t\:(pg.329).$$
Assume we have $\sigma^2(x)>0\:\forall x\in\mathbb{R}$ (non-degeneracy), and $\forall x\in\mathbb{R}\exists \epsilon>0$ st. $\int_{x-\epsilon}^{x+\epsilon}\frac{\lvert b(y)\rvert dy}{\sigma^2(y)}<\infty$ (local integrability). Now, let the scale function $p(x)$ be defined as:
$$p(x):= \int_c^x\exp{\left(-2\int_c^\xi\frac{b(\zeta)d\zeta}{\sigma^2(\zeta)}\right)}d\xi \:\:\:(pg.339).$$
Now, we define $v(x):=\int^x_cp'(y)\int_c^y\frac{2dz}{p'(z)\sigma^2(z)}dy\:\:(pg.347).$ Now let $I:=(l,r)$ denote the interval of interest and $\tau$ be the hitting time for our end points $l$ or $r$. Then, Feller's Test states that when $v(l^+)=v(r^-)=\infty$, $\prob(\tau=\infty)=1$ and when $v(l^+)=v(r^-)=\infty$ is not true, $\prob(\tau=\infty)<1$. 
\end{theorem}
Now, let us show two examples of applying Feller's test to show explosion of a process and non-explosion of a process. We will start with an explosive example with constant intensity noise. 
\begin{example}Explosiveness of Constant Intensity Noise. Consider the following SDE 
$$dX_t=X_t^2dt+\sigma dW_t$$
where $\sigma\in\mathbb{R}$ is a constant. It has been shown in Leon-Villa's Theorem 3.1, that such process explodes in finite time if and only if the Osgood criterion ($\int_a^\infty \frac{1}{b(t)}dt<\infty$) is satisfied \cite{leon2011osgood}. Since $\int_a^\infty\frac{1}{t^2}dt<\infty$, we expect Feller's test to state $\prob(\tau=\infty)<1$ or at least one of $v(\infty),\:v(-\infty)$ to be bounded. Let $\l^+=-\infty$ and $r^-=\infty$, so we have $I=(-\infty,\infty)=\mathbb{R}$. Note our $p(x)=\int_c^x\exp{(-2\int_c^\xi\frac{\zeta^2d\zeta}{\sigma^2})}d\xi$ and so we know our $v(x)$ is of the form
$$v(x)=\int^x_ce^{-\frac{2}{3\sigma^2}(y^3-c^3)}\left(\frac{2}{\sigma^2}\int_c^ye^{\frac{2}{3\sigma^2}(\zeta^3-c^3)}d\zeta \right)dy$$
$$=\int^x_ce^{-\frac{2}{3\sigma^2}(y^3)}\left(\frac{2}{\sigma^2}\int_c^ye^{\frac{2}{3\sigma^2}(\zeta^3)}d\zeta \right)dy$$
Now, when we have $v(\infty)$, we have
\begin{align*}
    v(\infty)&=\int^\infty_ce^{-\frac{2}{3\sigma^2}(y^3)}\left(\frac{2}{\sigma^2}\int_c^ye^{\frac{2}{3\sigma^2}(\zeta^3)}d\zeta \right)dy\\
\end{align*}
Let us first address the integral inside the integral. Note that we can establish the following inequality for sufficiently large $y$
$$\int_c^ye^{\zeta^3}d\zeta=\int_c^{y-1}e^{\zeta^3}d\zeta+\int_{y-1}^y\frac{3\zeta^2}{3\zeta^2}e^{\zeta^3}d\zeta$$
Now, we apply integration by parts by letting $u=\frac{1}{\zeta^2}$ and $dv=3\zeta^2 e^{\zeta^3}d\zeta$. Thus, we have
\begin{align*}
    \int_c^y\frac{3\zeta^2 e^{\zeta^3}}{3\zeta^2}d\zeta&=\frac{e^{y^3}}{y^2}-\frac{e^{c^3}}{c^2}+\int_c^y\frac{2e^{\zeta^3}}{\zeta^3}d\zeta\\
    &\leq k\frac{e^{y^3}}{y^2}
\end{align*}
Thus we can state by reparameterizing $\tilde{y}^3=\frac{2}{3\sigma^2}y^3$
\begin{align} \label{eq:v-infinity}
    v(\infty)&=\int^\infty_ce^{-\frac{2}{3\sigma^2}(y^3)}\left(\frac{2}{\sigma^2}\int_c^ye^{\frac{2}{3\sigma^2}(\zeta^3)}d\zeta \right)dy \nonumber\\
    &\leq \int^\infty_ce^{-\tilde{y}^3}k' \frac{e^{\tilde{y}^3}}{\tilde{y}^2}d\tilde{y}\nonumber \\
    &=\int_c^\infty \frac{k}{\tilde{y}^2}d\tilde{y}\:\:\:\:<\infty
\end{align}


Thus, since $v(\infty)<\infty$ we know from Feller's test that $\prob(\tau_\infty<\infty)>0$ as desired. 
\end{example}

\begin{example} \label{ex:non-explo} Non-explosiveness of High Power Noise. Consider the following SDE
$$dX_t=X_t^2dt+X_t^5dW_t.$$
Then we know $p(x)=\int_c^x \exp{\left(-2\int_c^\xi \frac{\zeta^2}{\zeta^{10}}\right)d\zeta}d\xi$. Then we know
$$v(x)=k\int_c^xe^{k_2y^{-7}}\int_c^y\zeta^{-10}e^{-k_2\zeta^{-7}}d\zeta dy$$
We want to establish a lower estimate for $v(x)$ and show that they go to infinity as $x\rightarrow\pm\infty$. Let $c>1$. Then we know 
$$\int_c^y\zeta^{-10}e^{-k_2\zeta^{-7}}d\zeta\geq\int_c^y\zeta^{-15}e^{-k_2\zeta^{-7}}d\zeta$$
Now we apply u-substitution for the integral on the right with $u=x^{-7}$ to attain 
$$\int_c^y\zeta^{-15}e^{-k_2\zeta^{-7}}d\zeta=-\frac{1}{7k_2^2} \int e^uudu$$
Now we apply integration by parts and substitute to find
$$\int_c^y\zeta^{-15}e^{-k_2\zeta^{-7}}d\zeta=-\frac{-k_2c^7e^{-\frac{k_2}{y^7}}-c^7y^7e^{-\frac{k_2}{y^7}}+k_2y^7e^{-\frac{k_2}{c^7}}+c^7y^7e^{-\frac{k_2}{c^7}}}{7k_2^2c^7y^7}$$
$$=\frac{(c^7y^7+k_2c^7)e^{-k_2y^{-7}}}{7k_2^2c^7y^7}+k_3$$
Thus, we know
\begin{align*}
    v(x)&\geq k\int_c^xe^{k_2y^{-7}}\int_c^y\zeta^{-15}e^{-k_2\zeta^{-7}}d\zeta dy\\
    &=k\int_c^xe^{k_2y^{-7}} \left(\frac{(c^7y^7+k_2c^7)e^{-k_2y^{-7}}}{7k_2^2c^7y^7}+k_3\right)dy
\end{align*}
Thus as $x\rightarrow\pm\infty$ we can see that we are integrating an approximate constant over infinite length and thus $v(\infty)=v(-\infty)=\infty$ as desired. 
\end{example}


\begin{remark}One can see that by the Osgood criterion, the ODE $dX=X^2dt$ explodes as $\int\frac{1}{X^2}dX<\infty$. When a small amount of noise is added through $dX_t=X_t^2+\sigma dW_t$, we observe explosion. However, when the noise is sufficiently large enough, like $dX_t=X_t^2+X_t^5 dW_t$, we see that our SDE does not explode. In fact, as $X_t>>0$, we see that the diffusion component becomes more significant than the force component. However, for dimension less than 2, the noise component cannot cause explosion as a drunk man will always find his way home while a drunk bird may not.  
\end{remark}

\begin{remark} One benefit of Feller's test is that one can apply it in a systematic manner. However, one downside of Feller's test is that it is difficult to extend to higher dimensions. Note that local integrability, which is used to construct $p(x)$, provides the following constraint $\frac{|b(x)|}{\sigma^2(x)}$. Thus, in higher dimensions, where $b(x)$ is a vector and $\sigma(x)$ is a matrix, it is not obvious how one can extend this test. Furthermore, Feller's test does not explore all three cases of explosion as it provides no information about almost sure explosion. 
\end{remark}

Another method to test for positive probability of explosion is through the use of Lyapunov function as described by Chow and Khasminskii \cite{chow2014almost}. Unlike Feller's test for explosion, which was methodical but involved constructing multiple integral functions that could be difficult to compute and was limited to 1-Dimensional processes, the Lyapunov method involves finding a function $V$ that satisfies certain criterion involving its end behavior and generator, which is computed through its partial derivatives, which are often easier to evaluate than integral functions. 

\section{Almost Sure Non-explosion of SDE through Lyapunov Functions}
In the following proofs for Khasminskii's theorems using Lyapunov functions, we use the concept of generators to establish a connection between the expectation and the initial condition as used in Dynkin's formula \cite{oksendal2013stochastic}. Thus, in this section we establish how to find the generator of an SDE with and without jumps for the sake of completeness.
For a stochastic process $X_t\in\mathbb{R}^d$, we define its generator $L$ acting on $f$ as
$$Lf=\lim_{h\downarrow 0}\frac{1}{h}\E(f(X_{t+h})-f(X_t)|X_t)\cite{oksendal2019stochastic}$$
Now, for the first case without jumps, let $X_t$ be a continuous process that is a solution of $dX_t=b(t,X_t)dt+\sigma(t,X_t)dW_t$ and $f\in\mathcal{C}^{1,2}$. From Ito's lemma, we know 
$$df(X_t)=\left(b(t,X_t)\frac{\partial f}{\partial X_t}+\frac{1}{2}\sigma(t,X_t)^2\frac{\partial^2 f}{\partial X_t^2}\right)dt+\sigma(t,X_t)\frac{\partial f}{\partial X_t}dW_t$$
Thus, it follows that
$$f(X_t)-f(X_0)=\int_0^tb(u,X_u)\frac{\partial f}{\partial X_u}+\frac{1}{2}\sigma(u,X_u)^2\frac{\partial^2 f}{\partial X_u^2}du+\int_0^t\sigma(u,X_u)\frac{\partial f}{\partial X_u}dW_u$$
Now, by taking the expectation and recalling that Ito integrals are martingales, we observe 
\begin{align*}
    \lim_{t\downarrow0}\frac{1}{t}\E(f(X_t)-f(X_0))&=\lim_{t\downarrow0}\frac{1}{t}\E\left(\int_0^tb(u,X_u)\frac{\partial f}{\partial X_u}+\frac{1}{2}\sigma(u,X_u)^2\frac{\partial^2 f}{\partial X_u^2}du\right)\\
    &=\lim_{t\downarrow0}\frac{1}{t}\int_0^t\E\left(b(u,X_u)\frac{\partial f}{\partial X_u}+\frac{1}{2}\sigma(u,X_u)^2\frac{\partial^2 f}{\partial X_u^2}\right)du\\
    &=b(t,X_t)f'(X_t)+\frac{1}{2}\sigma(t,X_t)^2f''(X_t)
\end{align*}
Now, for the case with jumps, let $X_t$ be a solution of the SDE of the form $dX_t=b(t,X_t)dt+\sigma(t,X_t)dW_t+J(t,X_t)dN_t$. Note from Ito's lemma, we know
$$df(X_t)=\left(b(t,X_t)\frac{\partial f}{\partial X_t}+\frac{1}{2}\sigma(t,X_t)^2\frac{\partial^2 f}{\partial X_t^2}\right)dt+\sigma(t,X_t)\frac{\partial f}{\partial X_t}dW_t+(f(X_{t^-}+Y_t)-f(X_{t^-}))dN_t$$
Thus, it follows from the integral form that 
$$f(X_t)-f(X_0)=\int_0^t\left(b(t,X_t)\frac{\partial f}{\partial X_u}+\frac{1}{2}\sigma(u,X_u)^2\frac{\partial^2 f}{\partial X_u^2}\right)du+\int_0^t\sigma(u,X_u)\frac{\partial f}{\partial X_u}dW_u$$
$$+\int_0^t(f(X_{s^-}+Y_s)-f(X_{s^-}))dN_s$$
Now, note that by taking the expectation, we have
\begin{align*}
    \lim_{t\downarrow0}\frac{1}{t}\E(f(X_t)-f(X_0))&=b(t,X_t)f'(X_t)+\frac{1}{2}\sigma(t,X_t)^2f''(X_t)+\E(\int_0^t(f(X_{s^-}+Y_s)-f(X_{s^-}))dN_s)\\
    &=b(t,X_t)f'(X_t)+\frac{1}{2}\sigma(t,X_t)^2f''(X_t)+\E(\sum^{N_t}(f(X_{s^-}+Y_s)-f(X_{s^-})))\\
    &=b(t,X_t)f'(X_t)+\frac{1}{2}\sigma(t,X_t)^2f''(X_t)+\E^{N_t}(\E(\sum^{N_t}(f(X_{s^-}+Y_s)-f(X_{s^-}))|N_t))\\
    &=b(t,X_t)f'(X_t)+\frac{1}{2}\sigma(t,X_t)^2f''(X_t)+\E^{N_t}(N_t\E(f(X_{s^-}+Y_s)-f(X_{s^-})|N_t))\\
    &=b(t,X_t)f'(X_t)+\frac{1}{2}\sigma(t,X_t)^2f''(X_t)+\E^{N_t}(N_t\E(f(X_{s^-}+Y_s)-f(X_{s^-})))\\
    &=b(t,X_t)f'(X_t)+\frac{1}{2}\sigma(t,X_t)^2f''(X_t)+\lambda\E^{Y_t}(f(X_{s^-}+Y_s)-f(X_{s^-})))
\end{align*}
Thus, when $f$ is a polynomial, the jump component of the generator becomes the moments for $Y$, with respect to the jump size distribution. 

\begin{remark}
Note that when we have an $f$ depending explicitly on time, $f(t,X_t)$, we have a $\frac{\partial f}{\partial t}$ term in our force from our Ito formula. In the proofs for the stability theorems, we will choose our Lyapunov functions to be of the form $f(X_t)=V(X_t)$, but we will multiply it by $e^{-C(t-t_0)}$ so we will end up with a $f(t,X_t)=V(X_t)e^{-C(t-t_0)}$ and utilize the $\frac{\partial}{\partial t}$ present in the Ito formula. 
\end{remark}

Now, we will show how Khasminskii's Lyapunov function can be used to show almost sure non-explosion $\prob(\tau_\infty=\infty)=1$ 

\begin{theorem} Khasminskii's theorem 3.5 for almost sure non-explosion in \cite{khasminskii2012stochastic}. Let $\mathcal{D}\subset\mathbb{R}^d$ be a bounded open set with regular boundary $\partial \mathcal{D}$ and let $\mathcal{D}^c$ denote the complement of $\mathcal{D}$. Suppose our SDE parameters satisfy the Lipschitz condition on any compact subset of $\mathbb{R}^d$ for any $t\geq t_0$. Now, assume there exist a positive function $V(x)$ belonging to the class $\mathcal{C}^{2}$ and positive constants $C$ such that the following conditions hold 
$$LV\leq CV$$
$$V=\inf_{|x|>R}V(x)\rightarrow\infty\:\text{as }R\rightarrow\infty$$
Then it follows that $\prob(\tau_\infty<\infty)=0$. 
\end{theorem}
\begin{proof}Let $\tau_R:=\tau_R\wedge t$. Note, with the chain rule argument, we see
\begin{align*}
    L\{V(x)e^{c(t-t_0)}\}&=e^{-c(t-t_0)}LV(x)-CV(x)e^{-c(t-t_0)}\\
    &=e^{-c(t-t_0)}(LV(x)-CV(x))\\
    &\leq0 \:\:\:\text{since $e^{-c(t-t_0)}>0$ and $LV(x)\leq CV(x)$ for $x\in\mathcal{D}^c$}
\end{align*}
Now note from Dynkin's formula that we have 
$$\E(V(X_{\tau_R)}e^{-C(\tau_R-t_0)})=V(X_{t_0})+\E\left(\int_0^{\tau_R}LV(x)e^{-C(s-t_0)}ds\right)$$
$$\E(V(X_{\tau_R})e^{-C(\tau_R-t_0)})\leq\E(V(X_{\tau_R}))e^{-C(t-t_0)}\leq V(X_{t_0})$$
$$\E(V(X_{\tau_R}))e^{-C(t-t_0)}\leq V(X_{t_0})$$
$$\E(V(X_{\tau_R}))\leq e^{C(t-t_0)}V(X_{t_0})$$
Now, let us split our expectation into two cases where $t\leq\tau_R$ and  $t>\tau_R$. When $t>\tau_R$, we know from definition of hitting time that $X_t$ has reached $D^c$. Thus, our expectation on the left hand side becomes
$$\E(V(X_{\tau_R}))=\E(V(X_{\tau_R}))\ind_{\{\tau_R<t\}}+\E(V(X_{\tau_R}))\ind_{\{t<\tau_R\}}$$

Now note that $\tau_R<t$, we have $V(X_{\tau_R})\geq\inf_{|x|>R}V(x)$ by definition of hitting time. For $t<\tau_R$, we know $V\geq0$ from our positivity assumption. Thus, we have the following string of inequality
$$V(X_0)e^{C(t-t_0)}\geq \E(V({X_{\tau_R}}))\geq \inf_{|x|>R}V(x)\prob(\tau_R\leq t)+0\prob(\tau_R>t)$$
Thus, it follows that
$$\frac{V(X_0)e^{C(t-t_0)}}{\inf_{|x|>R}V(x)}\geq\prob(\tau_R\leq t)$$
Thus, as $R\rightarrow\infty$, we know from our condition $V\rightarrow\infty$ and thus $\prob(\tau_R\leq t)=0$ as desired. 
\end{proof}

\begin{remark}
In \cite{chow2014almost}, Khasminskii and Chow states how the following condition serves as a sufficient condition for non-explosion for a $d-$dimensional SDE with $x\in \mathbb{R}^d$ and $d\bold{X}_t=b(\bold{X}_t,t)dt+\sigma(\bold{X}_t,t)d\bold{W}_t$ with $A(t,x):=\sigma(\bold{X}_t,t)\sigma^T(\bold{X}_t,t)$:
$$\langle b(t,x),x\rangle+\frac{1}{2}tr(A(t,x))-\lambda_{\max}(A(t,x))\leq C\lVert x\rVert^2\ln(\lVert x\rVert^2)$$
Although Khasminskii indicates that the condition was derived based on $V(x)=\ln(\lVert x\rVert^2)$, the actual gritty calculation is omitted. We will show the calculation in this section. From \cite{oksendal2013stochastic}, it is well known that the multi variable generator is of the form 
$$LV=\sum_{i=1}^db_i(t,x)\frac{\partial V}{\partial x_i}+\frac{1}{2}\sum_{i,j=1}^da_{i,j}(t,x)\frac{\partial^2 V}{\partial x_i\partial x_j}$$
Let $V(x)=\ln(\lVert x\rVert^2)=\ln(x_1^2+\dots+x_d^2)$. Thus, it follows 
$$\frac{\partial V}{\partial x_i}=\frac{2x_i}{\lVert x\rVert^2}$$
$$\frac{\partial^2 V}{\partial x_i\partial x_j}=\frac{-2x_ix_j}{\lVert x\rVert^4}\:\:\:\:\text{for $i\neq j$}$$
$$\frac{\partial^2 V}{\partial x_i^2}=\frac{-4x_i^2}{\lVert x\rVert^4}+\frac{2}{\lVert x\rVert^2}$$
Thus, we have
\begin{align*}
    LV&=\frac{2\langle b(t,x),x\rangle}{\lVert x\rVert^2}+\frac{1}{2}\sum_{i, j=1}^d a_{i,j}(t,x)\frac{-4x_ix_j}{\lVert x\rVert^4}+\frac{1}{2}\sum_{i=1}^da_{ii}(t,x)\frac{2}{\lVert x\rVert^2}\\
    &=\frac{2\langle b(t,x),x\rangle}{\lVert x\rVert^2}-2\frac{\langle Ax,x\rangle}{\lVert x\rVert^4}+\frac{\text{Tr}(A)}{\lVert x\rVert^2}
\end{align*}
Now, note that $A=\sigma^T(x,t)\sigma(x,t)$ is a positive symmetric definite matrix. Thus, we know
$$0\leq \lambda_{\text{min}} \leq\dots\leq\lambda_{\text{max}}$$
It then follows that
$$\lambda_{\text{min}}\lVert x\rVert^2\leq\langle Ax,x\rangle\leq\lambda_{\text{max}}\lVert x\rVert^2$$
Thus, since $\langle Ax,x\rangle\leq\lambda_{\text{max}}\lVert x\rVert^2$, we know
\begin{align*}
    LV&=\frac{2\langle b(t,x),x\rangle}{\lVert x\rVert^2}-2\frac{\langle Ax,x\rangle}{\lVert x\rVert^4}+\frac{\text{Tr}(A)}{\lVert x\rVert^2}\\
    &\leq\frac{2\langle b(t,x),x\rangle}{\lVert x\rVert^2}-2\frac{\lambda_{\text{max}}}{\lVert x\rVert^2}+\frac{\text{Tr}(A)}{\lVert x\rVert^2}
\end{align*}
Thus, when $\langle b(t,x),x\rangle+\frac{1}{2}tr(A(t,x))-\lambda_{\max}(A(t,x))\leq C\lVert x\rVert^2\ln(\lVert x\rVert^2)$ holds, it follows that $LV\leq CV$ as desired. 
\end{remark}

\begin{example}
Consider the following one dimensional SDE $dX_t=X_t^pdt+\sigma dW_t$ where $\sigma\in\mathbb{R}$ is a constant. When $0<p<1$ is true, $\int_1^\infty\frac{1}{x^p}dx=\infty$ and our SDE remains non-explosive almost surely according to \cite{leon2011osgood}. We can also see from Chow and Khasminskii's condition that we have
$$x^px+\frac{1}{2}\sigma^2-\sigma^2<x^{p+1}\leq x^2\ln(x^2)$$
which is true since $\frac{x^{p+1}}{x^2\ln(x^2)}\rightarrow0$ as $x\rightarrow\infty$ by applying L'Hospital and noting $p<1$. 
\end{example}

\begin{example}
Consider the following 2 dimensional SDE with $p>2$  $d\bold{X}_t=\begin{pmatrix}\lVert X_t\rVert^p&0\\ 0&\lVert X_t\rVert^p\end{pmatrix}d\bold{W}_t$. Thus, $A=\begin{pmatrix}\lVert X_t\rVert^{2p}&0\\ 0&\lVert X_t\rVert^{2p}\end{pmatrix}$. Then, since we have a diagonal matrix, we know $\bold{X}_t$ is non-explosive from Khasminskii's condition
$$\frac{2\lVert x\rVert^{2p}}{2}-\lVert x\rVert^{2p}=0\leq C\lVert x\rVert^2\ln(\lVert x\rVert^2)$$
However, we can see that from dimension 3, our condition no longer holds despite being a martingale. 
\end{example}

Now, we will show the extension of Khasminskii's Theorem for jump processes. The extension for jump processes follows in an analogous manner to Khasminskii's theorem for non-explosion. The most notable difference is that our condition for Lyapunov function must also hold for the interior of $D$. 

\begin{theorem}Let our SDE with jumps be of the form $dX_t=b(X_t)dt+\sigma(X_t)dW_t+Y_tdN_t$ where $dN_t$ has constant poisson arrival. Let $\mathcal{D}\subset\mathbb{R}^d$ be a bounded open set with regular boundary $\partial \mathcal{D}$ and let $\mathcal{D}^c$ denote the complement of $\mathcal{D}$. Suppose our SDE parameters satisfy the Lipschitz condition on any compact subset of $\mathbb{R}^d$ for any $t\geq t_0$. Now, assume there exist a positive function $V(x)$ belonging to the class $\mathcal{C}^{2}$ and positive constants $C$ such that the following conditions hold 
$$LV\leq CV$$
$$V=\inf_{|x|>R}V(x)\rightarrow\infty\:\text{as }R\rightarrow\infty$$
Then it follows that $\prob(\tau_\infty<\infty)=0$. 
\end{theorem}

\begin{proof} Let $\tau_R:=\tau_R\wedge t$. We can use the same chain rule argument to note 
\begin{align*}
    L\{V(x)e^{c(t-t_0)}\}&=e^{-c(t-t_0)}LV(x)-CV(x)e^{-c(t-t_0)}\\
    &=e^{-c(t-t_0)}(LV(x)-CV(x))\\
    &\leq0 \:\:\:\text{since $e^{-c(t-t_0)}<0$ and $LV(x)\leq CV(x)$ for $x\in\mathcal{D}^c$}
\end{align*}
Then, we apply Dynkin's formula again to find 
$$\E(V(X_{\tau_R)}e^{-C(\tau_R-t_0)})=V(X_{t_0})+\E\left(\int_0^{\tau_R}LV(x)e^{-C(s-t_0)}ds\right)$$
$$\E(V(X_{\tau_R})e^{-C(\tau_R-t_0)})\leq\E(V(X_{\tau_R}))e^{-C(t-t_0)}\leq V(X_{t_0})$$
$$\E(V(X_{\tau_R}))e^{-C(t-t_0)}\leq V(X_{t_0})$$
$$\E(V(X_{\tau_R}))\leq e^{C(t-t_0)}V(X_{t_0})$$
Now, let us split our expectation into two cases where $t\leq\tau_R$ and  $t>\tau_R$. When $t>\tau_R$, we know from definition of hitting time that $X_t$ has reached $D^c$. Thus, our expectation on the left hand side becomes
$$\E(V(X_{\tau_R}))=\E(V(X_{\tau_R}))\ind_{\{\tau_R<t\}}+\E(V(X_{\tau_R}))\ind_{\{t<\tau_R\}}$$

Now note that $\tau_R<t$, we have $V(X_{\tau_R})\geq\inf_{|x|>R}V(x)$ by definition of hitting time. For $t<\tau_R$, we know $V\geq0$ from our positivity assumption. Thus, we have the following string of inequality
$$V(X_0)e^{C(t-t_0)}\geq \E(V({X_{\tau_R}}))\geq \inf_{|x|>R}V(x)\prob(\tau_R\leq t)+0\prob(\tau_R>t)$$
Thus, it follows that
$$\frac{V(X_0)e^{C(t-t_0)}}{\inf_{|x|>R}V(x)}\geq\prob(\tau_R\leq t)$$
Thus, as $R\rightarrow\infty$, we know from our condition $V\rightarrow\infty$ and thus $\prob(\tau_R\leq t)=0$ as desired. 
\end{proof}

\begin{remark}
    Note that our Lyapunov function cannot have points of singularities anymore. Thus, a potential candidate for our Lyapunov function is $V(x)=x^2$. Then, when we consider a Merton Jump Diffusion model with lognormal jump size as described by the following SDE
    $$dX_t=\mu X_tdt+\sigma X_tdW_t+X_t(Y_t-1)dN_t$$
    We can show stability with $V(x)=x^2$. Note $V(x)\rightarrow\infty$ as $|x|\rightarrow\infty$. Furthermore
    \begin{align*}
        LV&=b(t,X_t)V'(X_t)+\frac{1}{2}\sigma(t,X_t)^2V''(X_t)+\lambda( X_t^2\E(Y_t^2)-X_t^2)\\
        &=2\mu X_t^2+\sigma^2X_t^2+\lambda(e^{2\mu_J+2\sigma^2_J}-1)X_t^2\\
        &=(2\mu+\sigma^2+\lambda(e^{2\mu_J+2\sigma^2_J}-1))X_t^2\\
        &\leq (2\mu+\sigma^2+\lambda(e^{2\mu_J+2\sigma^2_J}-1))V(X_t)
    \end{align*}
    Thus, we have found our $C$ as desired.  
\end{remark}

\begin{theorem}
Corollary 3.1 from Khasminskii\cite{khasminskii2012stochastic}. Assume the conditions from theorem 3.1 hold. Now, let $\mathcal{D}$ be a set with regular boundaries $\partial\mathcal{D}$. Let $D_n$ be an increasing sequence such that $\cup^\infty \mathcal{D}_n=\mathcal{D}$. Assume $V(t,x)$ be twice differentiable with respect to $x$ and once with respect to $t$. Let $X_0\in \mathcal{D}$. 
\end{theorem}
\begin{proof} The proof follows analogously as before but with slight modifications to our hitting time conditions. We will provide an example to illuminate the details, but for the sake of brevity the proof will be omitted. 
\end{proof}

\begin{example}
Consider an SDE of the form: $dX_t=X_t^{-\alpha}dt+dW_t$ where $\alpha>1$. Note that when $x=0$, $b(x)=x^{-\alpha}=\infty$. Now, let $D=(0,\infty)$, and $D_n=(\frac{1}{n},\infty)$. Let $V(x)=\frac{1}{x}$. Then $V'(x)=-x^{-2}$ and $V''(x)=2x^{-3}$. Note:
\begin{align*}
    LV&=-x^{-2-\alpha}+x^{-3}\\
    &=x^{-1}(x^{-2}-x^{1-\alpha})
\end{align*}
Note since $\alpha>1$, we know $1+\alpha>2$. Thus:
$$\lim_{x\rightarrow0}(x^{-2}-x^{-1-\alpha})=-\infty$$
$$\lim_{x\rightarrow\infty}(x^{-2}-x^{-1-\alpha})=0$$
$$\sup_{x>0}(x^{-2}-x^{-1-\alpha})=C<\infty$$
Note, since we have $LV\leq CV$, we can use the chain rule argument to note 
$$\E(X_t)^{-1}\leq X_0^{-1}e^{Ct}$$
Now let $\tau_n:=\inf\{t>0:X_t\leq\frac{1}{n}\}$. Then it follows: 
$$\E(X_{t\wedge\tau_n})^{-1}\leq X_0^{-1}e^{Ct\wedge\tau_n}$$
Now when $\tau_n\leq t$, we know $X_{t\wedge \tau_n}=\frac{1}{n}$. Thus, we have
$$n\prob(\tau_n\leq t)\leq X_0^{-1}e^{Ct}$$
$$\prob(\tau_n\leq t)\leq\frac{e^{Ct}}{nX_0}$$
Now, as $n\rightarrow\infty$, we know $\prob(\tau_n\leq t)=0$ as desired. Thus, we have shown that $X_t$ cannot become negative or explode to infinity and will always reside within $D=(0,\infty)$.
\end{example}

\section{Positive Probability of Explosion Through Lyapunov Functions}
\begin{theorem} Chow and Khasminskii's Theorem 1. Let $\mathcal{D}\subset\mathbb{R}^d$ be a bounded open set with regular boundary $\partial \mathcal{D}$ and let $\mathcal{D}^c$ denote the complement of $\mathcal{D}$. Suppose our SDE parameters satisfy the Lipschitz condition on any compact subset of $\mathbb{R}^d$ for any $t\geq t_0$. Now, assume there exist a positive function $V(t,x)$ belonging to the class $\mathcal{C}^{1,2}([t_0,\infty]\times\mathcal{D}^c)$ and positive constants $K_1,\:K_2,\:K_3$ and $C$ such that the following conditions hold 
$$\sup_{t\geq t_0,x\in\mathcal{D}^c}V(x)=K_1<\infty$$
$$\sup_{t\geq t_0,x\in\mathcal{D}}V(x)=K_2<\inf_{t\geq t_0x\in\Gamma}V(x)=K_3$$
for some set $\Gamma\subset\mathcal{D}^c$, and 
$$LV(x)\geq CV(x)$$
for $t\geq t_0,\: x\in\mathcal{D}^c$. Then the solution $X_t$ with the initial condition $X_{t_1}=x\in\Gamma$ for some $t_1\geq t_0$ has an explosion with positive probability. 
\end{theorem}
\begin{proof} Let $\tau_R:=\min\{\tau_{R(t)},t,\tau_{\partial \mathcal{D}}\}$ (the process stops once it hits the boundary of $\mathcal{D}$ or leaves the disk $R$). The following is intended to serve as a more intuitive explanation of the proof provided by Khasminskii and Chow. Consider a disk $R$ of radius $r$, a disk inside $D$ in $R$ and a disk $\Gamma$ encircling $D$ (for our construction, $\Gamma\cap D=\emptyset$) as shown below. 
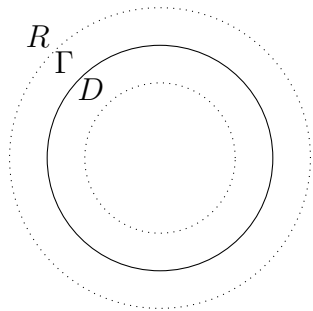
\begin{center}
    \begin{tikzpicture}
    \draw[dotted](0,0)circle(2cm); 
    \node at(135:2.3cm){$R$};
    \draw (0,0) circle(1.5cm); 
    \node at(135:1.8cm) {$\Gamma$};
    \draw[dotted](0,0)circle(1cm);
    \node at(135:1.3cm){$D$};
\end{tikzpicture}
\end{center}
Let our process start from $\Gamma$ and our goal is to show there is a positive probability of reaching the boundary of $R$. Note, for $x\in\mathcal{D}^c$, we have by applying the product rule 
\begin{align*}
    L\{V(x)e^{c(t-t_0)}\}&=e^{-c(t-t_0)}LV(x)-CV(x)e^{-c(t-t_0)}\\
    &=e^{-c(t-t_0)}(LV(x)-CV(x))\\
    &\geq0 \:\:\:\text{since $e^{-c(t-t_0)}>0$ and $LV(x)\geq CV(x)$ for $x\in\mathcal{D}^c$}
\end{align*}
Now, from Dynkin's formula, we find 
$$\E(V(X_{\tau_{R}})e^{-C(\tau_R-t_0)})=V(X_{t_0})+\E\left(\int_0^{\tau_R}LV(x)e^{-C(t-t_0)}ds\right)$$
Thus
$$\E(V(\tau_R)e^{-C(\tau_R-t_0)})\geq V(X_{t_0}).$$
Now, we can divide our expectation into three cases: the case where we have hit $R$, the case where we have hit $\partial D$, and the case where we have hit neither $R$ or $\partial D$. This splitting of expectation can be denoted in the following manner 
$$\E(V(X_{\tau_R})e^{-C(\tau_R-t_0)})\ind_{\{\tau_R<t\wedge\tau_{\partial D}\}}+\E(V(X_{\tau_R})e^{-C(\tau_R-t_0)})\ind_{\{\tau_{\partial D}<t\wedge\tau_{R}\}}$$
$$+\E(V(X_t)e^{-C(t-t_0)})\ind_{\{t<\tau_R\wedge\tau_{\partial D}\}}\geq V(X_{t_0})$$
Now from the conditions on our Lyapunov function, we note 
\begin{align*}
    \E(V(X_{\tau_R})e^{-C(\tau_R-t_0)})\ind_{\{\tau_{\partial R}<t\wedge\tau_{\partial D}\}}&\geq V(X_{t_0})-K_2-K_1e^{-C(t-t_0)}\\
    &\geq K_3-K_2-K_1e^{-C(t-t_0)}
\end{align*}
Now, as $R\rightarrow\infty,\:t\rightarrow\infty$, we note
$$\E(K_1e^{-C(\tau_\infty-t_0)})\geq K_3-K_2>0$$
Thus, it follows that $\prob(\tau_\infty<\infty)>0$ since $\prob(\tau_\infty=\infty)=1$ would imply that our expectation on the left is zero.
\end{proof}

\begin{remark} Analogous to Remark 2.3, Khasminskii and Chow provides the condition for positive probability for explosion in \cite{chow2014almost}
$$\langle b(t,x),x\rangle+\frac{1}{2}\text{Tr}(A(t,x))-\lambda_{\text{max}}\left(1+\frac{1+\epsilon}{\ln(\lVert x\rVert^2)}\right)\geq C\lVert x\rVert^2(\ln(\lVert x\rVert^2))^{1+\epsilon}$$
The derivation follows the exact same steps as shown in Remark 2.3 but with $V(x)=K-\frac{1}{\ln(\lVert x\rVert)^\epsilon}$ and will thus be skipped. 
\end{remark}
\begin{example}
    Consider our 3 dimensional extension of example 2.4. Let our SDE be of the form: $d\bold{X}_t=\begin{pmatrix}\lVert X_t\rVert^p&0&0\\ 0&\lVert X_t\rVert^p&0\\ 0&0&\lVert X_t\rVert^p\end{pmatrix}d\bold{W}_t$, where $p>1$. Note $A=\begin{pmatrix}\lVert X_t\rVert^{2p}&0&0\\ 0&\lVert X_t\rVert^{2p}&0\\ 0&0&\lVert X_t\rVert^{2p}\end{pmatrix}$. Then, from remark 3.2, we need to show there exists $C$ and $\epsilon$ such that 
    $$\frac{3}{2}\lVert x\rVert^{2p}-\lVert x\rVert^{2p}\left(1+\frac{1+\epsilon}{\ln(\lVert x\rVert^2)}\right)\geq C\lVert x\rVert^2(\ln(\lVert x\rVert^2)^{1+\epsilon}$$
    holds for all $\lVert x\rVert>R$. Note that it suffices to check 
    $$\frac{\ln(\lVert x\rVert^2)^{1+\epsilon}}{\lVert x\rVert^{p-1}}\rightarrow0\:\text{as $\lVert x\rVert\rightarrow\infty$}$$
    Note by properties of exponents and L'Hospital, we have 
    $$\frac{2^{1+\epsilon}\frac{(1+\epsilon)\ln(\lVert x\rVert)}{\lVert x\rVert}}{\frac{p-1}{\lVert x\rVert^{2-p}}}\equiv \frac{\ln(\lVert x\rVert)}{\lVert x\rVert^{p-1}}\rightarrow 0$$
    Thus, despite being a martingale, our process without drift can explode with positive probability from the noise alone in dimension 3. 
\end{example}

Now, let us extend theorem 3.1 to jump processes 

\begin{theorem} Extending Theorem 3.1 to Jump Processes. Let our SDE be of the form  $dX_t=b(X_t,t)dt+\sigma(X_t,t)dW_t+Y_tdN_t$ where $Y_t$ is the jump size and $N_t$ is poisson with constant intensity. Now, let $\mathcal{D}\subset\mathbb{R}^d$ be a bounded open set with regular boundary $\partial \mathcal{D}$ and let $\mathcal{D}^c$ denote the complement of $\mathcal{D}$. Suppose our SDE parameters satisfy the Lipschitz condition on any compact subset of $\mathbb{R}^d$ for any $t\geq t_0$. Now, assume there exist a positive function $V(t,x)$ belonging to the class $\mathcal{C}^{1,2}([t_0,\infty]\times\mathcal{D}^c)$ and positive constants $K_1,\:K_2,\:K_3$ and $C$ such that the following conditions hold  
$$\sup_{t\geq t_0,x\in\mathcal{D}^c}V(x)=K_1<\infty$$
$$\sup_{t\geq t_0,x\in\mathcal{D}}V(x)=K_2<\inf_{t\geq t_0x\in\Gamma}V(x)=K_3$$
for some set $\Gamma\subset\mathcal{D}^c$, and 
$$LV(x)\geq CV(x)$$
for $t\geq t_0,\: x\in\mathcal{D}^c$.
\end{theorem}
\begin{proof} Note that even with the presence of jumps, since our choice of $V$ depends solely on $x$, we can still apply the product rule argument used in theorem 1.1 
\begin{align*}
    L\{V(x)e^{c(t-t_0)}\}&=e^{-c(t-t_0)}LV(x)-CV(x)e^{-c(t-t_0)}\\
    &=e^{-c(t-t_0)}(LV(x)-CV(x))\\
    &\geq0 \:\:\:\text{since $e^{-c(t-t_0)}>0$ and $LV(x)\geq CV(x)$ for $x\in\mathcal{D}^c$}
\end{align*}
However, since our process can jump over the boundary of $D$, we simply change our hitting time condition to be the whole set $D$ and the proof follows analogously. From Dynkin's formula, we find  
$$\E(V(X_{\tau_{R}})e^{-C(\tau_R-t_0)})=V(X_{t_0})+\E\left(\int_0^{\tau_R}LV(x)e^{-C(t-t_0)}ds\right)$$
Thus 
$$\E(V(\tau_R)e^{-C(\tau_R-t_0)})\geq V(X_{t_0}).$$
Now, we can divide our expectation into three cases: the case where we have hit $R$, the case where we have hit $\partial D$, and the case where we have entered neither $R$ or $D$. This splitting of expectation can be denoted in the following manner 
$$\E(V(X_{\tau_R})e^{-C(\tau_R-t_0)})\ind_{\{\tau_R<t\wedge\tau_{ D}\}}+\E(V(X_{\tau_R})e^{-C(\tau_R-t_0)})\ind_{\{\tau_{ D}<t\wedge\tau_{R}\}}$$
$$+\E(V(X_t)e^{-C(t-t_0)})\ind_{\{t<\tau_R\wedge\tau_{D}\}}\geq V(X_{t_0})$$
Now from the conditions on our Lyapunov function, we note 
\begin{align*}
    \E(V(X_{\tau_R})e^{-C(\tau_R-t_0)})\ind_{\{\tau_R<t\wedge\tau_{ D}\}}&\geq V(X_{t_0})-K_2-K_1e^{-C(t-t_0)}\\
    &\geq K_3-K_2-K_1e^{-C(t-t_0)}
\end{align*}
Now, as $R\rightarrow\infty,\:t\rightarrow\infty$, we note
$$\E(K_1e^{-C(\tau_\infty-t_0)})\geq K_3-K_2>0$$
Thus, it follows that $\prob(\tau_\infty<\infty)>0$ since $\prob(\tau_\infty=\infty)=1$ would imply that our expectation on the left is zero.
\end{proof}

\section{Almost Sure Explosion}
In Khasminskii and Chow, a theorem containing the sufficient condition for almost sure explosion using Lyapunov functions is provided. In this section, we will provide an intuitive explanation of their proof and provide a sufficient condition for almost sure explosion in 1 dimension, namely ellipticity. 

\begin{theorem} Khasminskii and Chow's Almost Sure Explosion \cite{chow2014almost}. Suppose the conditions for theorem 3.1 hold. Assume
$$\inf_{t\geq t_0,x\in\mathcal{D}^c}V(t,x)=K_0>0$$
and for $x\in\partial D$
$$\prob(\tau_\Gamma<\infty)=1$$
Then for any $t\geq t_0,x\in\mathbb{R}^n$, $\prob(\tau_\infty<\infty)=1$. 
\end{theorem}
\begin{proof}
    The proof can be found in \cite{chow2014almost}.
\end{proof}
\begin{remark}
The main additional constraint is having $\prob(\tau_\Gamma<\infty)=1$. Note, because we are assuming the conditions for theorem 2.2, we know that when our process reaches $\Gamma$, there is a positive probability of reaching the outer disk $R$ in finite time, or $\prob(\tau_\infty<\infty)>0$. Thus, our process starts from $\Gamma$ and either returns to $\Gamma$ in a finite amount of time or hits $\partial\mathcal{D}$ and returns to $\Gamma$ in a finite amount of time with probability 1. Thus, every time our process returns to $\Gamma$ we have a positive probability of explosion, and so eventually our process must explode as you cannot have an infinite coin toss without observing a heads.  
\end{remark} 

A sufficient criterion for $\prob(\tau_\Gamma<\infty)=1$ is ellipticity\cite{koralov2007theory}, which for one dimensional SDE of the form $dX_t=b(x,t)dt+\sigma(x,t)dW_t$, has $\sigma(x,t)>c_0>0$. Or in short, the noise cannot be killed. It is worth mentioning that ellipticity in multiple dimension is defined by the following condition 
$$v^T\sigma(x)\sigma(x)^Tv\geq c_0\lVert v\rVert^2$$
where $\sigma$ is a matrix and $v$ is a vector. In short, as long as the noise stays bounded above zero, our process cannot get stuck inside $D$ and must reach $\Gamma$ in a finite amount of time as desired. 

\begin{example}Leon and Villa has already shown that for $b(X_t)>0$ and increasing with $\sigma(t,X_t)=\sigma$, where $\sigma$ is a constant, we have almost sure explosion when $\int_1^\infty\frac{1}{b(y)}dy<\infty$ and almost sure no-explosion when $\int_1^\infty\frac{1}{b(y)}dy=\infty$\cite{leon2011osgood}. Let us further assume that $\frac{b'(x)}{b(x)^2}\rightarrow0$ as $x\rightarrow\infty$. Note that for the almost sure explosion case, we can choose $V(x)=\int_1^x \frac{1}{b(y)}dy$. Then $\lim_{x\rightarrow\infty}V(x)=K_1<\infty$ and thus our supremum and infimum conditions are satisfied. Furthermore, since $b(x)$ is positive monotonic increasing, we know $\sup_{x\in\partial D}V(x)<\inf_{x\in\Gamma}V(x)$. Now, we need to check that there exists a positive $C$ such that $LV\geq CV$. Note
$$LV=\frac{b(x)}{b(x)}-\frac{\sigma}{2}\frac{b'(x)}{b(x)^2}$$
Thus, from our assumption, we know for a sufficiently large $x$, we know our $LV>0$. Furthermore, since $V(x)=\int_1^x\frac{1}{b(y)}dy<K_1$, we can choose always choose $C$ to be sufficiently smaller than our $LV$, so our condition $LV\geq CK_1\geq CV$ also holds. 
\end{example}

\begin{remark}
For the cases when $b(x)$ is a simple polynomial like $b(x)=x^p$, we know $\int^\infty_1\frac{1}{x^p}dx<\infty$ when $p>1$. Then, we know $\frac{b'(x)}{b(x)^2}=\frac{(p-1)x^{p-1}}{x^{2p}}\rightarrow0$ as $x\rightarrow\infty$. Thus, it is worth noticing that the added assumption $\frac{b'(x)}{b(x)^2}$ does not provide much hindrance for many of the cases one may find in nature. 
\end{remark}

\section{Conclusion}
Thus, in this paper, we started our expedition of finite time explosion for SDEs through ODE comparisons as presented by Leon and Villa, which stated that explosion for SDEs with constant intensity noise follows analogously with the explosion of ODEs through the Osgood criterion. However, from the prominent Feller's test for explosion, we were able to observe that for different types of noise functions, we could have almost sure non-explosion and positive probability of explosion regardless of the Osgood criterion. Then, we explored Khasminskii's Lyapunov function methods to classify explosion for SDE in multiple dimensions and provide an extension to jump processes. Nevertheless, because of the one-dimensional aspect of Feller's test for explosion and the lengthy integration tests, we adopted Khasminskii's method of Lyapunov functions to characterize explosions for SDEs. 

We first establish Khasminskii's conditions for almost sure non-explosion using Khasminskii's Lyapunov method. We then show the almost sure non-explosion for constant intensity is consistent with Lyon Villa's finding with the Osgood criterion as shown in Example 2.4. We show further examples on a 2-dimensional Ito integral that will be compared with the 3-dimensional Ito integral to illuminate the dependence on dimension with explosions. 

We then use Khasminskii's theorem on almost sure explosion and show that a sufficient condition for explosion is ellipticity, which in one dimension is consistent with the findings of Leon and Villa with the Osgood criterion. Afterwards, we show through Khasminskii's criterion for positive probability of explosion which is derived through Khasminskii's Lyapunov method for positive probability of explosion. We then show that the 3-dimensional Ito integral, which has zero mean, can explode with positive probability.

\section*{Acknowledgement}
I am grateful to Professor M. Salins for introducing me to the problem of finite time explosion in SDEs and the Osgood criterion. His super power of glancing at a paper once and explaining it in simple terms helped me choose the most illuminating examples to include. 

\section{Appendix}

We show that an Ito integral is a martingale with zero mean. Note from definition of Ito integral $I_n$
\begin{align*}
    \E(I_{n+1}|\mathcal{F}_n)&=\E\left(\sum_{k=0}^{n}\Delta_k(W_{n+1}-W_n)\bigg{|}\mathcal{F}_n\right)\\
    &=\E\left(\sum_{k=0}^{n-1}\Delta_k(W_{n}-W_{n-1})+\Delta_n(W_{n+1}-W_{n})\bigg{|}\mathcal{F}_n\right)\\
    &=I_n+\E(\Delta_n(W_{n+1}-W_{n})|\mathcal{F}_n)\\
    &=I_n \:\:\:\:\text{as $\E(W_{n+1}-W_{n}|\mathcal{F}_n)=0$}
\end{align*}
Thus, we find the Ito integral is a martingale as desired.

\bibliographystyle{plain}
\bibliography{bib}

\end{document}